\documentclass[10pt,psamsfonts]{amsart}
\usepackage{amsmath}
\usepackage{amsthm}
\usepackage{amssymb}
\usepackage{amscd}
\usepackage{amsfonts}
\usepackage{amsbsy}
\usepackage[latin1]{inputenc}
\usepackage{epsfig,afterpage}
\usepackage[dvips]{psfrag}
\usepackage{psfrag}
\usepackage[all]{xy}
\usepackage{color}

\newcommand{\e}{\varepsilon}

\usepackage{overpic}

\newtheorem {theorem} {Theorem} 
\newtheorem {proposition} [theorem] {Proposition}

\newtheorem {remark} {Remark}

\begin{document}

\title[Zero-Hopf bifurcation in a Chua system]
{Zero-Hopf bifurcation in a Chua system}

\author[R.D. Euz\'{e}bio and J. Llibre]
{Rodrigo D. Euz\'{e}bio$^{1,2}$ and Jaume Llibre$^2$}

\address{$^1$ Departament de Matem\'atica,
IBILCE, UNESP, Rua Cristovao Colombo 2265, Jardim Nazareth, CEP
15.054-00, Sao Jos\'e de Rio Preto, SP, Brazil}
\email{rodrigo.euzebio@sjrp.unesp.br}

\address{$^2$ Departament de Matem\`{a}tiques,
Universitat Aut\`{o}noma de Barcelona, 08193 Bellaterra, Barcelona,
Catalonia, Spain} \email{jllibre@mat.uab.cat}

\subjclass[2010]{37G15, 37G10, 34C07}

\keywords{Chua system, periodic orbit, averaging theory, zero Hopf
bifurcation}

\date{}
\dedicatory{}

\maketitle

\begin{abstract}
A zero-Hopf equilibrium is an isolated equilibrium point whose
eigenvalues are $\pm \omega i\neq 0$ and $0$. In general for a such
equilibrium there is no theory for knowing when from it bifurcates
some small--amplitude limit cycle moving the parameters of the
system. Here we study the zero-Hopf bifurcation using the averaging
theory. We apply this theory to a Chua system depending on $6$
parameters, but the way followed for studying the zero-Hopf
bifurcation can be applied to any other differential system in
dimension $3$ or higher.

In this paper first we show that there are three $4$-parameter
families of Chua systems exhibiting a zero-Hopf equilibrium. After,
by using the averaging theory, we provide sufficient conditions for
the bifurcation of limit cycles from these families of zero-Hopf
equilibria. From one family we can prove that $1$ limit cycle
bifurcate, and from the other two families we can prove that $1$,
$2$ or $3$ limit cycles bifurcate simultaneously.
\end{abstract}

\section{Introduction and statement of the main results}\label{introduction}

The Chua system is a classical model for electronic circuit and one
of the most simplest models presenting chaos. It was presented by
Leon Chua in 1983 and exhibit a rich range of dynamical behaviour.
Precisely, the Chua circuit is a relaxation oscillator with a cubic
nonlinear characteristic. It can be though as a circuit comprising a
harmonic oscillator for which the operation is based on a
field-effect transistor, coupled to a relaxation oscillator composed
of a tunnel diode. The Chua system can be described by the following
equations
\begin{equation}\label{chuasystem}
\begin{array}{rcl}
\dfrac{dx}{dt}&=&a(z-bx-a_{2}x^2-a_{1}x^3),\vspace{0.2cm}\\
\dfrac{dy}{dt}&=&-z,\vspace{0.2cm}\\
\dfrac{dz}{dt}&=&-b_{1}x+y+b_{2}z.
\end{array}
\end{equation}
Note that it depends on six parameters $a$, $a_1$, $a_2$, $b$, $b_1$
and $b_2$.

\smallskip

In \cite{LV} the autors analyze the existence of local and global
analytic first integrals in the Chua system. In \cite{RG} the
authors use techniques of Differential Geometry in order to obtain
an analytical expression of the slow manifold equation of Chua
system. In \cite{M} was studied the dynamics at infinity of the Chua
system for the particular case where $b_{1}$ and $b_{2}$ are booth
one. Besides, we can found some aspects about the Hopf bifurcation
in \cite{BMM} and \cite{AMSL}. In this paper, by using averaging
theory, we study the limit cycles that can bifurcate from
\eqref{chuasystem} from zero-Hopf equilibrium points of the Chua
system. We recall that at these points we cannot apply the classical
Hopf bifurcation theory which needs that the real eigenvalue be
non-zero.

\smallskip

The Chua system can have at most three equilibria, namely: the
origin and the two equilibria
$$
p_{\pm}=\left(\dfrac{-a_{2}\pm\sqrt{a_{2}^{2}-4a_{1}b}}{2a_{1}},
-\dfrac{a_{2}b_{1}}{2a_{1}}\pm\dfrac{b_{1}\sqrt{a_{2}^{2}-
4a_{1}b}}{2a_{1}},0\right),
$$
if $a_{2}^{2}-4a_{1}b>0$ and $a_1\neq 0$. When $a_{2}^{2}-4a_{1}b=0$
and $a_1 a_2\neq 0$ the system has only two equilibra, the origin
and the equilibrium
$$
p=\left(\dfrac{-a_{2}}{2a_{1}}, -\dfrac{a_{2}b_{1}}{2a_{1}}
,0\right).
$$
Otherwise the origin is the unique equilibrium of the system.

\smallskip

As far as we know, the study of existence or non-existence of
zero-Hopf equilibria and zero-Hopf bifurcation in the Chua system
have not been considered in the literature. In this paper we have
this objective. The method used here for studying the zero-Hopf
bifurcation can be applied to any differential system in
$\mathbb{R}^{3}$. In fact, this method already has been used in
order to study the zero-Hopf bifurcations of the R\"{o}ssler
differential system, see \cite{L}.

\smallskip

A {\it zero-Hopf equilibrium} is an equilibrium point of a
$3-$dimensional autonomous differential system which has a zero
eigenvalue and a pair of purely imaginary eigenvalues. In general
the zero-Hopf bifurcation is a $2-$parameter unfolding of a
$3-$dimensional autonomous differential system with a zero-Hopf
equilibrium. The unfolding has an isolated equilibrium with a  zero
eigenvalue and a pair of purely imaginary eigenvalue if the two
parameters take zero values and the unfolding has different dynamics
in a small neighborhood of this isolated equilibrium as the two
parameters vary in a small neighborhood of the origin. To read more
about zero-Hopf bifurcation, see Guckenheimer, Han, Holmes,
Kuznetsov, Marsden and Scheurle in \cite{G}, \cite{GH}, \cite{H},
\cite{K} and \cite{SM}. Moreover, complex phenomena can occur at an
isolated zero-Hopf equilibrium, as bifurcation of complicated
invariant sets of the unfolding and a local birth of ``chaos", as
can be seen in the work of Baldom\'{a} and Seara, Broer and Vegter,
Champneys and Kirk, Scheurle and Marsden in \cite{BS}, \cite{BS2},
\cite{BV}, \cite{CK} and \cite{SM}.

\smallskip

In the next proposition we characterize the Hopf equilibria of the
Chua system.

\begin{proposition}\label{familyorigin}
There are three $4$-parameter families of Chua systems having a
zero-Hopf equilibrium point, one for the equilibrium point located
at the origin and the other two for each one of the equilibria $p_{\pm}$
when they exist. Namely,
\begin{itemize}
\item[(a)] $b=b_{2}=0$ and $a b_{1}+1>0$ for the origin; and

\item[(b)] $b=a_{2}^{2}/(4a_{1})$,  $b_{2}=0$, $a b_{1}+1>0$,
$a_{2}^{2}-4a_{1}b>0$ and $a_1\neq 0$ for $p_{\pm}$.
\end{itemize}
\end{proposition}

The next result gives sufficient conditions for the bifurcation of a
limit cycle from the origin when it is a zero-Hopf equilibrium.

\begin{theorem}\label{maintheoremorigin}
Let
$$
(a,a_{1},a_{2},b,b_{1},b_{2})=(\overline{a}_{0}+\varepsilon
\alpha_{0},\overline{a}_{1}+\varepsilon
\alpha_{1},\overline{a}_{2}+\varepsilon \alpha_{2},\varepsilon
\beta_{0},\dfrac{\omega^{2}-1}{a}+\varepsilon \beta_{1},\varepsilon
\beta_{2}).
$$
If $\overline{a}_{0}\overline{a}_{2}\neq0$, $|\omega|\neq 0,1$ and
$$
\Gamma=(\overline{a}_{0}\beta_{0}(1-\omega^{2})+\beta_{2}\omega^{2})
(\overline{a}_{0}\beta_{0}\omega^{2}(1-\omega^{2})+\beta_{2}
\omega^{4})>0,
$$
then for $\e>0$ sufficiently small the Chua system has a zero-Hopf
bifurcation at the equilibrium point located at the origin of
coordinates, and a limit cycle borns at this equilibrium when
$\varepsilon=0$. Moreover, this limit cycle has the same kind of
stability or instability than an equilibrium point of a planar
differential system with eigenvalues
\begin{equation}\label{eigenvalues}
\dfrac{-\beta_{2}\omega^{5}\pm\sqrt{\omega^{6}(\beta_{2}^{2}
\omega^{4}(3-2\omega^{2})+2\overline{a}_{0}^{2}\beta_{0}^{2}
(\omega^2-1)^{3})}}{2\omega^{6}(\omega^2-1)}.
\end{equation}
\end{theorem}

The following result provides sufficient conditions for the
bifurcation of a limit cycle from the equilibrium $p_-$ when it is
zero-Hopf equilibrium.

\begin{theorem}\label{maintheoremp1}
Consider the vector $(a,a_{1},a_{2},b,b_{1},b_{2})$ given by
\begin{equation}
\begin{array}{l}
a=\overline{a}_{0}+\varepsilon\alpha_{0}+\varepsilon^{2}\xi_{0},
\vspace{0.2cm}\\
a_{1}=\overline{a}_{1}+\varepsilon\alpha_{1}+\varepsilon^{2}\xi_{1},
\vspace{0.2cm}\\
a_{2}=\varepsilon\alpha_{2}+\varepsilon^{2}\xi_{2},\vspace{0.2cm}\\
b=\dfrac{a_{2}^{2}}{4a_{1}}+\varepsilon^{2}\zeta_{0},\vspace{0.2cm}\\
b_{1}=\dfrac{\omega^2-1}{a}+\varepsilon\beta_{1}+\varepsilon^{2}
\zeta_{1},\vspace{0.2cm}\\\label{conditionsp1}
b_{2}=\varepsilon^{2}\zeta_{2}.
\end{array}
\end{equation}
If $\overline{a}_{1}\omega\neq0$ and $\overline{a}_{1}\zeta_{0}<0$
then, for $\e>0$ sufficiently small the Chua system has a zero-Hopf
bifurcation at the equilibrium point located at $p_{-}$ and three
limit cycles can bifurcate from this equilibrium when
$\varepsilon=0$. Moreover, examples of systems where $1$, $2$ or $3$
limit cycles bifurcate simultaneously are given.
\end{theorem}

Proposition \ref{familyorigin} and Theorems \ref{maintheoremorigin}
and \ref{maintheoremp1} are proved in section \ref{proofs}. In
particular, booth theorems are proved using the averaging method.
This method will be briefly summarized in the next section. We note
that Theorem \ref{maintheoremorigin} is proved using averaging
theory of first order, but the proof of Theorem \ref{maintheoremp1}
needs averaging of second order.

\smallskip

Also the stability or instability of the bifurcated limit cycles in
Theorem \ref{maintheoremp1} can be studied, but the expressions of
the eigenvalues which provide such stability or instability are huge
and we do not give them here.

\begin{remark}
For the equilibrium point $p_{+}$ we have analogous results to the
ones of Theorem \ref{maintheoremp1} for $p_-$. For this reason, we
omit the statement of the result for the equilibrium $p_+$ and its
proof.
\end{remark}

\section{Limit cycles via averaging theory}\label{averaging}

The averaging method is a classical tool in nonlinear analysis and
dynamical systems. The procedure of averaging can be found already
in the work of Lagrange and Laplace who provided an intuitive
justification of the method. After them, Poincar\'{e} considered the
determination of periodic solutions by series expansion with respect
to a small parameter and around 1930 we see the start of precise
statements and proofs in averaging theory. After this time many new
results in the theory of averaging have been obtained. The main
contribution in direction to the formalization of the method is due
to Fatou \cite{F}. The work of Krylov and Bogoliubov \cite{BK} and
Bogoliubov \cite{B} also provide important practical and theoretical
improvements in the theory.

\smallskip

Now we present the basic results on the averaging theory of first
and second order. The averaging of first order for studying periodic
orbits can be found in \cite{SV}, see Theorems 11.5 and 11.6. It can
be summarized as follows.

\begin{theorem}\label{averaging1order}
We consider the following two initial value problems
\begin{equation}
\dot{x}=\varepsilon f(t,x)+\varepsilon^{2}g(t,x,\varepsilon),
\qquad x(0)=x_{0},\label{originalsystem}
\end{equation}
and
\begin{equation}\label{averagedsystem}
\dot{y}=\varepsilon f^{0}(y),\qquad y(0)=x_{0}.
\end{equation}
where $x,y,x_{0}\in\Omega$ an open subset of $\mathbb{R}^{n}$,
$t\in[0,\infty)$, $\varepsilon\in(0,\varepsilon_{0}]$, $f$ and $g$
are periodic of period $T$ in the variable $t$, and $f^{0}(y)$ is
the averaged function of $f(t,x)$ with respect to $t$, i.e.,
\begin{equation}\label{average}
f^{0}(y)=\displaystyle\dfrac{1}{T}\int_{0}^{T}f(t,y)dt.
\end{equation}

Suppose:
\begin{itemize}
\item[(i)] $f$, its Jacobian $\dfrac{\partial f}{\partial x}$, its
Hessian $\dfrac{\partial^{2} f}{\partial x^{2}}$, $g$ and its
Jacobian $\dfrac{\partial g}{\partial x}$ are defined, continuous
and bounded by a constant independent on $\varepsilon$ in
$[0,\infty)\times\Omega$ and $\varepsilon\in(0,\varepsilon_{0}]$;

\item[(ii)] $T$ is a constant independent of $\varepsilon$; and

\item[(iii)] $y(t)$ belongs to $\Omega$ on the interval of time
$[0,1/\varepsilon]$. Then the following statements hold.
\begin{itemize}
\item[(a)] On the time scale $1/\varepsilon$ we have that
$x(t)-y(t)=\mathit{O}(\varepsilon)$, as $\varepsilon\to0$.

\item[(b)] If $p$ is a singular point of the averaged system
\eqref{averagedsystem} such that the determinant of the Jacobian
matrix
\begin{equation}\label{jacobian}
\left.\dfrac{\partial f^{0}}{\partial y}\right|_{y=p}
\end{equation}
is not zero, then there exists a limit cycle $\phi(t,\varepsilon)$
of period $T$ for system \eqref{originalsystem} which is close to
$p$ and such that $\phi(0,\varepsilon)\to p$ as $\varepsilon\to0$.

\item[(c)] The stability or instability of the limit cycle $\phi(t,
\varepsilon)$  is given by the stability or instability of the
singular point $p$ of the averaged system \eqref{averagedsystem}. In
fact, the singular point $p$ has the stability behaviour of the
Poincar\'{e} map associated to the limit cycle $\phi(t,\varepsilon)$.
\end{itemize}
\end{itemize}
\end{theorem}

The next result present the second order averaging method of a
periodic differential system. For a proof see Theorem 3.5.1 of
Sanders and Verhulst in \cite{SV}, see also \cite{BL}.

\begin{theorem}\label{averagingsecond order}
We consider the following two initial value problems
\begin{equation}\label{averaging2order}
\dot{x} = \varepsilon f(t, x) + \varepsilon^2 g(t,x) + \varepsilon^3
R(t,x,\varepsilon), \quad x(0) = x_0
\end{equation}
and
\begin{equation}\label{averaging2order1}
\dot{y} = \varepsilon f^{0}(y) + \varepsilon^2( f^{10}(y) +
g^{0}(y)), \quad y(0) = x_0,
\end{equation}
with $f, g: [0, \infty) \times \Omega \rightarrow \mathbb {R}^{n}$,
$R: [0, \infty) \times \Omega \times (0, \varepsilon_0] \rightarrow
\mathbb {R}^{n}$, $\Omega$ an open subset of $\mathbb {R}^{n}$, $f,
g$ and $R$ periodic of period $T$ in the variable $t$,
$$
f^{1}(t,x) = \dfrac{\partial f}{\partial x}y^{1}(t,x),\quad
\mbox{where} \quad y^{1}(t,x)=\displaystyle\int_{0}^{t}f(s,x)ds.
$$

Of course, $f^{0}$, $f^{10}$ and $g^{0}$ denote the averaged
functions of $f$, $f^{1}$ and $g$, respectively, defined as in
\eqref{average}. Suppose:
\begin{itemize}
\item[(i)] $\partial f/\partial x$ is Lipschitz in $x$, $g$ and $R$ are
Lipschitz in $x$ and all these functions are continuous on their
domain of definition;

\item[(ii)] $|R(t,x,\varepsilon)|$ is bounded by a constant uniformly in
$[0,L/\varepsilon)\times\Omega\times(0,\varepsilon_{0}]$;

\item[(iii)] $T$ is a constant independent of $\varepsilon$; and

\item[(iv)] $y(t)$ belongs to $\Omega$ on the interval of time
$[0,1/\varepsilon]$. Then the following statements hold.
\begin{itemize}
\item[(a)] In the time scale $1/\varepsilon$ we have that $x(t)=
y(t)+\varepsilon y^{1}(t,y(t))=\mathit{O}(\varepsilon^{2})$.

\item[(b)] If $f^0(y)\equiv 0$ and $p$ is a singular point of averaged
system \eqref{averaging2order1} such that
$$
\left.\dfrac{\partial(f^{10}+g^{0})(y)}{\partial y}\right|_{y=p}
$$
is not zero, then there exist a limit cycle $\phi(t,\varepsilon)$ of
period $T$ for system \eqref{averaging2order} which is close to $p$
and such that $\phi(0,\varepsilon)\to p$ as $\varepsilon\to0$.

\item[(c)] The stability or instability of the limit cycle
$\phi(t,\varepsilon)$  is given by the stability os instability of
the singular point $p$ of the averaged system
\eqref{averaging2order1}. In fact, the singular point $p$ has the
stability behaviour of the Poincar\'{e} map associated to the limit
cycle $\phi(t,\varepsilon)$.
\end{itemize}
\end{itemize}
\end{theorem}

\section{Proofs}\label{proofs}

In this section we give the proofs of the results presented in
section \ref{introduction}.

\begin{proof}[Proof of Proposition \ref{familyorigin}]
The characteristic polynomial of the linear part of the Chua system
at the origin is
$$
p(\lambda)=-\lambda^{3}+(b_{2}-ab)\lambda^{2}+(b_{2}ab-ab_{1}-1)
\lambda-ab.
$$
Imposing that $p(\lambda)=-\lambda(\lambda^{2}+\omega^{2})$, we
obtain $b=b_{2}=0$ and $b_{1}=(\omega^{2}-1)/a$. So statement (a)
follows.

\smallskip

The characteristic polynomial of the linear part of the Chua system
at $p_{-}$ is given by
$$
p(\lambda)=-\dfrac{(1-b_{2}\lambda+\lambda^{2})[2a_{1}\lambda+
a(a_{2}^{2}+a_{2}\sqrt{a_{2}^{2}-4a_{1}b}-4a_{1}b)]+2a_{1}b_{1}
a\lambda}{2a_{1}}
$$
The proposition follows imposing that $p(\lambda)=-\lambda(
\lambda^{2}+ \omega^{2})$, and that the equilibrium point $p_-$
exists.
\end{proof}

\begin{proof}[Proof of Theorem \ref{maintheoremorigin}]
If we consider
$$
(a,a_{1},a_{2},b,b_{1},b_{2})=(\overline{a}_{0}+\varepsilon
\alpha_{0},\overline{a}_{1}+\varepsilon
\alpha_{1},\overline{a}_{2}+\varepsilon \alpha_{2},\varepsilon
\beta_{0},\dfrac{\omega^{2}-1}{a}+\varepsilon \beta_{1},\varepsilon
\beta_{2}).
$$
with $\varepsilon>0$ a sufficiently small parameter, then the Chua
system becomes
\begin{equation}\label{systemnew}
\begin{array}{rcl}
\dot{x}&=&(\overline{a}_{0}+\varepsilon\alpha_{0})(\varepsilon\beta_{0}
x+z-(\overline{a}_{2}+\varepsilon\alpha_{2})x^{2}-(\overline{a}_{1}+
\varepsilon\alpha_{1})x^{3}),\vspace{0.2cm}\\
\dot{y}&=&-z,\vspace{0.2cm}\\
\dot{z}&=&-\left(\varepsilon\beta_{1}+\dfrac{\omega^2-1}{\overline{a}_{0}
+\varepsilon\alpha_{0}}\right)x+y+\varepsilon\beta_{2} z.
\end{array}
\end{equation}

By the rescaling of variables $(x,y,z)=(\varepsilon X,\varepsilon
Y,\varepsilon Z)$, system \eqref{systemnew} becomes
\begin{equation}\label{systemnewvariables}
\begin{array}{rcl}
\dot{X}&=&(\overline{a}_{0}+\varepsilon\alpha_{0}) (\varepsilon\beta_{0}
X+Z-\varepsilon(\overline{a}_{2}+\varepsilon\alpha_{2})X^{2}-\varepsilon^{2}
(\overline{a}_{1}+\varepsilon\alpha_{1})X^{3}),\vspace{0.2cm}\\
\dot{Y}&=&-Z,\vspace{0.2cm}\\
\dot{Z}&=&-\left(\varepsilon\beta_{1}+\dfrac{\omega^2-1}{\overline{a}_{0}
+\varepsilon\alpha_{0}}\right)X+Y+\varepsilon\beta_{2} Z.
\end{array}
\end{equation}

Now we shall write the linear part at the origin of
\eqref{systemnewvariables} into its real Jordan normal form
\begin{equation}
\left(
\begin{array}{ccc}
0&-\omega&0\\
\omega&0&0\\
0&0&0
\end{array}\label{jordan}
\right),
\end{equation}
when $\e=0$. For doing that we do the linear change of variables
$(X,Y,Z)\rightarrow(u,v,w)$ given by
\begin{equation}\label{changevariables}
\begin{array}{rcl}
X&=&\dfrac{\overline{a}_{0}(w+\omega v)}{\omega^{2}},\vspace{0.2cm}\\
Y&=&w-\dfrac{w}{\omega^{2}}-\dfrac{v}{\omega},\vspace{0.2cm}\\
Z&=&u.
\end{array}
\end{equation}

In these new variables, system \eqref{systemnewvariables} is written
as follow
\begin{equation}\label{systemuvw}
\begin{array}{rl}
\dot{u}=&-v\omega         +           \varepsilon\dfrac{-(\alpha_{0}
(1-\omega^{2})+\overline{a}_{0}^{2}\beta_{1})(w+\omega
v)+\overline{a}_{0}
u\beta_{2}\omega^{2}}{\overline{a}_{0}\omega^{2}}\vspace{0.2cm}\\
&-\varepsilon^{2}\dfrac{\alpha_{0}^{2}(\omega^2-1)(w+\omega v)}
{\overline{a}_{0}^{2}\omega^{2}},\vspace{0.2cm}\\
\dot{v}=&u\omega      -       \varepsilon\dfrac{     (\omega^2-1)
(-u\alpha_{0}\omega^{4}+\overline{a}_{0}^{2}(w+\omega v)(\beta_{0}
\omega^{2}+\overline{a}_{0}\overline{a}_{2}(w+\omega v)))     }
{   \overline{a}_{0}\omega^{5}   }\vspace{0.2cm}\\
&-\varepsilon^{2}\dfrac{1}{\omega^{7}}(\omega^2-1)(w+\omega v)
(\alpha_{0}\beta_{0}\omega^{4}+\overline{a}_{0}(w+\omega v)
(\overline{a}_{2}\alpha_{0}\omega^{2}\vspace{0.2cm}\\
& +\overline{a}_{0}\alpha_{2}\omega^{2}+\overline{a}_{0}^{2}
\alpha_{1} (w+\omega v))),\vspace{0.2cm}\\
\dot{w}=&-\varepsilon\dfrac{     -u\alpha_{0}\omega^{4}+
\overline{a}_{0}^{2}(w+\omega
v)(\beta_{0}\omega^{2}+\overline{a}_{0} \overline{a}_{2}(w+\omega
v))     }{   \overline{a}_{0}\omega^{4}   }
\vspace{0.2cm}\\
&-      \varepsilon^{2}\dfrac{  (w+\omega v)(\alpha_{0}\beta_{0}
\omega^{4}+\overline{a}_{0}(w+\omega v)(\overline{a}_{2}\alpha_{0}
\omega^{2}+\overline{a}_{0}\alpha_{2}\omega^{2}))   }{   \omega^{6}
}.
\end{array}
\end{equation}

Writing the differential system \eqref{systemuvw} in cylindrical
coordinates $(r,\theta,w)$ by $u=r\cos\theta$, $v=r\sin\theta$ and
$w=w$ we have
\begin{equation}\label{systempolar}
\begin{array}{rcl}
\dfrac{dr}{d\theta}&=&\varepsilon\left(\dfrac{ r \beta_{2} \cos
^2\theta }{\omega}
-\dfrac{\overline{a}_{0}(\omega^2-1)\sin\theta(w+r\omega\sin\theta)
(\overline{a}_{0}\overline{a}_{2}w+\beta_{0}\omega^{2}}{\omega^{6}}\right.
\vspace{0.2cm}\\
&&\dfrac{+\overline{a}_{0}\overline{a}_{2}rw\sin\theta)}{\omega^{6}}
-\dfrac{\cos\theta(w(\alpha_{0}+\overline{a}_{0}^{2}\beta_{1}-\alpha_{0}\omega^{2})
+rw  (  \overline{a}_{0}\beta_{1}  }{\overline{a}_{0}\omega^{3}}\vspace{0.2cm}\\
&&\left.\dfrac{-2\alpha_{0}(\omega^2-1)   )\sin\theta )
}{\overline{a}_{0}
\omega^{3}}\right)+\mathit{O}(\varepsilon^{2}),\vspace{0.2cm}\\
\dfrac{dw}{d\theta}&=&\varepsilon    \dfrac{r\alpha_{0}\omega^{4}\cos\theta
-\overline{a}_{0}^{2}(w+r\omega\sin\theta)(\overline{a}_{0}\overline{a}_{2}w
+\beta_{0}\omega^{2}+\overline{a}_{0}\overline{a}_{2}r\omega\sin\theta) }
{\overline{a}_{0}\omega^{5}}\vspace{0.2cm}\\
&&+\mathit{O}(\varepsilon^{2}).
\end{array}
\end{equation}

Now we apply the first order averaging theory as described in
Theorem \ref{averaging1order} of section \ref{averaging}. In order
to do this, we note that \eqref{systempolar} satisfies all the
assumptions of Theorem \ref{averaging1order}, where we identify
$t=\theta$, $T=2\pi$, $x=(r,w)^{T}$,
$F(\theta,r,w)=(F_{1}(\theta,r,w),F_{2}(\theta,r,w))$ and
$f(r,w)=(f_{1}(r,w),f_{2}(r,w))$.

\smallskip

By calculating $f_{1}$ and $f_{2}$, we get
$$
\begin{array}{rcl}
f_{1}(r,w)&=&\dfrac{1}{2\pi}\displaystyle\int_{0}^{2\pi}F_{1}
(\theta,r,w)d\theta\vspace{0.2cm}\\
&=&\dfrac{r(\beta_{2}\omega^{4}-2\overline{a}_{0}^{2}
\overline{a}_{2}w(\omega^2-1)-\overline{a}_{0}\beta_{0}\omega^{2}
(\omega^2-1))}{2\omega^{5}},\vspace{0.2cm}\\
f_{2}(r,w)&=&\dfrac{1}{2\pi}\displaystyle\int_{0}^{2\pi}F_{2}
(\theta,r,w)d\theta\vspace{0.2cm}\\
&=&-\dfrac{\overline{a}_{0}(      2w\beta_{0}\omega^{2}+
\overline{a}_{0}\overline{a}_{2}(2w^{2}+r^{2}\omega^{2})
)}{2\omega^{5}}.
\end{array}
$$

There is only one solution $(r^{*},w^{*})$ for
$f_{1}(r,w)=f_{2}(r,w)=0$ satisfying $r^{*}>0$ and this solution is
$$
\begin{array}{rcl}
r^{*}&=& \sqrt{   \dfrac{\Gamma}{2\overline{a}_{0}^{4}
\overline{a}_ {2}^{2}(\omega^2-1)^{2}}    },\vspace{0.2cm}\\
w^{*}&=&\dfrac{\overline{a}_{0}\beta_{0}\omega^{2}(1-\omega^{2})
+\beta_{2}\omega^{4}}{2\overline{a}_{0}^{2}\overline{a}_
{2}(\omega^2-1)},
\end{array}
$$
since $\overline{a}_{0}\overline{a}_{2}\neq0$, $|\omega|\neq1$ and
$\Gamma>0$.

\smallskip

We note that the Jacobian \eqref{jacobian} at $(r^{*},w^{*})$ takes
the value
$$
\dfrac{\beta_{2}^{2}\omega^{4}-\overline{a}_{0}^{2}\beta_{0}^{2}
(\omega^2-1)^{2}}{2\omega^{6}(\omega^2-1)}
$$
and the eigenvalues of the Jacobian matrix
$$
\left.\dfrac{\partial(f_{1},f_{2})}{\partial(r,w)}\right|_{(r,w)
=(r^{*},w^{*})}=\left(
\begin{array}{cc}
0&-\dfrac{1}{\sqrt{2}\omega^{5}}\sqrt{\overline{a}_{0}
(\omega^2-1)\Gamma}\\
-\dfrac{1}{\omega^{3}}\sqrt{\dfrac{\overline{a}_{0}\Gamma}{2
(\omega^2-1)^{3}}}&\dfrac{\beta_{2}}{\omega(1-\omega^{2})}
\end{array}
\right)
$$
are the ones given in \eqref{eigenvalues}.

\smallskip

In short, from Theorem \ref{averaging1order} we conclude the proof
once we show that periodic solutions corresponding to
$(r^{*},w^{*})$ provides a periodic solution bifurcating from the
origin of coordinates of the differential system \eqref{systemnew}
when $\varepsilon=0$. Theorem \ref{averaging1order} guarantees for
$\varepsilon>0$ sufficiently small the existence of a periodic
solution corresponding to the point $(r^{*},w^{*})$ of the form
$(r(\theta,\varepsilon),w(\theta,\varepsilon))$ such that
$(r(0,\varepsilon),w(0,\varepsilon))\to(r^{*},w^{*})$ when
$\varepsilon\to0$. So system \eqref{systemuvw} has a periodic
solution
\begin{equation}\label{periodicorbit}
(u(\theta,\varepsilon)=r(\theta,\varepsilon)\cos\theta,v(\theta,
\varepsilon)=r(\theta,\varepsilon)\sin\theta,w(\theta,\varepsilon))
\end{equation}
for $\varepsilon>0$ sufficiently small. Consequently, from relation
\eqref{periodicorbit} through the linear change of variables
\eqref{changevariables} system \eqref{systemnewvariables} has a
periodic solution $(X(\theta),Y(\theta),Z(\theta))$. Finally, for
$\varepsilon>0$ sufficiently small system \eqref{systemnew} has a
periodic solution $(x(\theta),y(\theta),z(\theta))=(\varepsilon
X(\theta),\varepsilon Y(\theta),\varepsilon Z(\theta))$ which tends
to the origin of coordinates when $\varepsilon\to0$. Thus, it is a
periodic solution starting at the zero-Hopf equilibrium point
located at the origin of coordinates when $\varepsilon=0.$ This
completes the proof of theorem.
\end{proof}

Since the proof of Theorem \ref{maintheoremp1} is very similar to
the of Theorem \ref{maintheoremorigin}, then we will omit some steps
in order to avoid some long expressions.

\begin{proof}[Proof of Theorem \ref{maintheoremp1}]
Suppose that we have the conditions given in \eqref{conditionsp1} on
the parameters of Chua system \eqref{chuasystem}. Then, by a
translation of the equilibrium point $p_-$ at the origin of
coordinates, and a rescaling of variables given by
$(x,y,z)=(\varepsilon X,\varepsilon Y,\varepsilon Z)$ the Chua
system becomes
\begin{equation}\label{systemnewp1}
\begin{array}{rcl}
\dot{X}&=&A_{1}X+(\overline{a}_{0}+\alpha_{0}\varepsilon+\varepsilon^{2}
\xi)Z+A_{2}X^{2}+A_{3}X^{3},\vspace{0.2cm}\\
\dot{Y}&=&-Z,\vspace{0.2cm}\\
\dot{Z}&=&A_{4}X+Y+\varepsilon^{2}\zeta_{2}Z,
\end{array}
\end{equation}
where
$$
\begin{array}{rcl}
A_{1}&=&\varepsilon^{2}(1/2\overline{a}_{1}^{2})(\overline{a}_{0}
(-4\overline{a}_{1}^{2}\zeta_{0}-\alpha_{1}\alpha_{2}\varepsilon
\sqrt{-\overline{a}_{1}\zeta_{0}}+
2\overline{a}_{1}\sqrt{-\overline{a}_{1}\zeta_{0}}(\alpha_{2}+
\varepsilon\xi_{2})\vspace{0.2cm}\\
&&+2\overline{a}_{1}\alpha_{0}\varepsilon(-2\overline{a}_{1}
\zeta_{0}+\alpha_{2}\sqrt{-\overline{a}_{1}\zeta_{0}})),
\vspace{0.2cm}\\
A_{2}&=&\varepsilon^{2}(1/2\overline{a}_{1})(\overline{a}_{0}
(3\alpha_{1}\varepsilon\sqrt{-\overline{a}_{1}\zeta_{0}}+
\overline{a}_{1}(\alpha_{2}+6\sqrt{-\overline{a}_{1}\zeta_{0}}+
\varepsilon\xi_{2})     )\vspace{0.2cm}\\
&&+\overline{a}_{1}\alpha_{0}\varepsilon(\alpha_{2}+6\sqrt{
-\overline{a}_{1}\zeta_{0}})),\vspace{0.2cm}\\
A_{3}&=&\varepsilon^{2}(\overline{a}_{1}\alpha_{0}\varepsilon+
\overline{a}_{0}(\overline{a}_{1}+\alpha_{1}\varepsilon)),
\vspace{0.2cm}\\
A_{4}&=&(\omega^2-1)[\overline{a}_{0}^{3}-\alpha_{0}^{3}
\varepsilon^{3}-\overline{a}_{0}^{2}\varepsilon(\alpha_{0}+
\varepsilon\xi_{0}+
\overline{a}_{0}\alpha_{0}\varepsilon^{2}(\alpha_{0}+2\varepsilon\xi_{0}))]
\vspace{0.2cm}\\
&&+\overline{a}_{0}^{4}\varepsilon(\beta_{1}+\varepsilon\zeta_{1}).
\end{array}
$$

The linear part of \eqref{systemnewp1} at $p_{-}$ in the real Jordan
normal form when $\varepsilon=0$ is given by \eqref{jordan}, and
doing also the linear change of variables $(X,Y,Z)\rightarrow
(u,v,w)$ given by \eqref{changevariables} we write the linear part
of system \eqref{systemnewp1} in its real Jordan normal form when
$\e=0$, we obtain the system
\begin{equation}\label{systemuvwp1}
\begin{array}{rl}
\dot{u}=& -\omega v+ \e (B_{1}v+B_{2}w) + \e^{2}\zeta_{2}u,\vspace{0.2cm}\\
\dot{v}=& \omega u +\dfrac{\varepsilon\alpha_{0}(\omega^2-1)u}
{\overline{a}_{0}\omega} +\varepsilon\dfrac{\omega^2-1}
{\overline{a}_{0}\omega}B_{3},\vspace{0.2cm}\\
\dot{w}=&\dfrac{\varepsilon\alpha_{0}}{\overline{a}_{0}}u+\dfrac{
\varepsilon^{2}}{\overline{a}_{0}}B_{3},
\end{array}
\end{equation}
where
$$
\begin{array}{rcl}
B_{1}&=&-\dfrac{\overline{a}_{0}^{3}(\beta_{1}+
\zeta_{1})-\overline{a}_{0}
(\alpha_{0}+\varepsilon\xi_{0})(\omega^2-1)-\e
\alpha_{0}^{2}(\omega^2-1)}{\overline{a}
_{0}^{2}\omega},\vspace{0.2cm}\\
B_{2}&=&-\dfrac{\overline{a}_{0}^{3}(\beta_{1}+
\zeta_{1})-\overline{a}_{0} (\alpha_{0}+\varepsilon\xi_{0})
(\omega^2-1)+\e\alpha_{2}^{2}(\omega^2-1)}{\overline{a}_{0}^{2}
\omega^{2}},\vspace{0.2cm}\\
B_{3}&=&\xi_{0}u-\dfrac{\overline{a}_{0}^{2}(w+\omega v)(2(-2\overline{
a}_{1}\zeta_{0}+\alpha_{2}\sqrt{\overline{a}_{1}\zeta_{0}})\omega^{4}
-\overline{a}_{0}\overline{a}_{1}(\alpha_{2}}{2\overline{a}_{1}\omega^{6}}
\vspace{0.2cm}\\
&&\dfrac{+6\omega^{2}(w+\omega v)\sqrt{\overline{a}_{1}\zeta_{0}})
+2\overline{a}_{0}^{2}\overline{a}_{1}^{2}(w+\omega v)^{2})}
{2\overline{a}_{1}\omega^{6}}.
\end{array}
$$

If we write system \eqref{systemuvwp1} in cylindrical coordinates
$(r,\theta,w)$ defined by $u=r\cos\theta$, $v=r\sin\theta$ and
$w=w$, after we take as new independent variable the angle $\theta$,
and we apply to the system  $dr/d\theta$ and $dw/d\theta$ that we
obtain the second order averaging method described in Theorem
\ref{averagingsecond order}, we get that the function
$f=(f_{1},f_{2})$ is identically zero, and that the function
$g=(g_{1},g_{2})$ is
$$
\begin{array}{rcl}
g_{1}(r,w)&=&\dfrac{\pi r}{4\omega}\left(4\zeta_{2}+\dfrac{
\overline{a}_{0}(\omega^2-1)(4\overline{a}_ {0}\overline{a}_{1}
w(\alpha_{2}+6\sqrt{-\overline{a}_{1}\zeta_{0}})\omega^{2}}
{\overline{a}_{1}\omega^{6}}\right.\vspace{0.2cm}\\
&&\left.+\dfrac{4(2\overline{a}_{1}\zeta_{0}-\alpha_{2}
\sqrt{-\overline{a}_{1}\zeta_{0}})\omega^{4}-3\overline{a}_{0}^{2}
\overline{a}_{1}^{2}(4w^{2}+3r^{2}\omega^{2}))}{\overline{a}_{1}
\omega^{6}}\right),\vspace{0.2cm}\\
g_{2}(r,w)&=&\dfrac{\overline{a}_{0}\pi}{2\overline{a}_{1}
\omega^{7}}\left(     4w(2\overline{a}_{1}\zeta_{0}-\alpha_{2}
\sqrt{-\overline{a}_{1}\zeta_{0}})\omega^{4}-2\overline{a}_{0}^{2}
\overline{a}_{1}^{2}w(2w^{2}+
3r^{2}\omega^{2})      \right.\vspace{0.2cm}\\
&&\left.   +\overline{a}_{0}\overline{a}_ {1}(\alpha_{2}+6\sqrt{
-\overline{a}_{1}\zeta_{0}})\omega^{2}(2w^{2}+r^{2}\omega^{2})
\right).
\end{array}
$$

In order to find solutions $(r^{*},w^{*})$ of $g=0$ we compute a
Gr\"{o}bner basis $\{b_{k}(r,w)$, $k=1,\ldots, 20\}$ in the variables
$r$ and $w$ for the set of polynomials $\{\overline{g}_{1}(r,w),
\overline{g}_{2}(r,w)\}$ where $\overline{g}_{1}= 4(\overline{a}_{1}
\omega^{7}/\pi r)g_{1}$ and $\overline{g}_{2}= (2\overline{a}_{1}
\omega^{7}/\overline{a}_{0}\pi)g_{2}$ and then we will look for
roots of $b_{1}$ and $b_{2}$. It is a known fact that the solutions
of a Gr\"{o}bner basis of $\{\overline{g}_{1}(r,w), \overline{g}_{2}
(r,w)\}$ are the solutions of $\overline{g}_{1}=0$ and $\overline{g}
_{2}=0$, consequently solutions of $g_{1}=0$ and $g_{2}=0$ as well.
For more information about Gr\"{o}bner basis see \cite{AL} and
\cite{L2}.

\smallskip

The Gr\"{o}bner basis for the polynomials $\{\overline{g}_{1}(r,w),
\overline{g}_{2}(r,w)\}$ in the variables $r$ and $w$ is formed by
twenty polynomials. We only use two polynomials of this basis,
namely,
$$
\begin{array}{rcl}
G_{1}(r,w)&=&30(\omega^2-1)\overline{a}_{0}^{4}\overline{a}_{1}^{3}w^{3}-
15(\omega^2-1)\omega^{2}\overline{a}_{1}^{2}\overline{a}_{0}^{3}
(\alpha_ {2}+6\sqrt{-\overline{a}_{1}\zeta_{0}})w^{2}\vspace{0.2cm}\\
&&+2\overline{a}_{0}\overline{a}_{1}\omega^{4}(-6\overline{a}_{1}
\zeta_{2}\omega^{2}+\overline{a}_{0}(\alpha_ {2}^{2}-42\overline{a}_{1}
\zeta_{0}\vspace{0.2cm}\\
&&+15\alpha_{2}\sqrt{-\overline{a}_{1}\zeta_{0}}(\omega^2-1)))w+2
\omega^{6}(\overline{a}_{1}(\alpha_{2}+6\sqrt{-\overline{a}_{1}
\zeta_{0}}\zeta_{2}\omega^{2})\vspace{0.2cm}\\
&&+\overline{a}_{0}(8\overline{a}_{1}\alpha_{2}\zeta_{0}-\alpha_{2}^{2}
\sqrt{-\overline{a}_{1}\zeta_{0}}-12(-\overline{a}_{1}\zeta_{0})
\sqrt{-\overline{a}_{1}\zeta_{0}}(\omega^2-1)))
\end{array}
$$
and
$$
\begin{array}{rcl}
G_{2}(r,w)&=&\overline{a}_{0}\overline{a}_{1}\omega^{2}(6\overline{a}_{0}
\overline{a}_{1}w-\alpha_{2}\omega^{2}-6\omega^{2}\sqrt{-\overline{a}_{1}
\zeta_{0}})r^{2}+2w(2\overline{a}_{0}^{2}
\overline{a}_{1}^{2}w^{2}\vspace{0.2cm}\\
&&-\overline{a}_{0}\overline{a}_{1}w(\alpha_{2}+6\sqrt{-\overline{a}_{1}
\zeta_{0}})\omega^{2}+2(-2\overline{a}_{1}\zeta_{0}+
\alpha_{2}\sqrt{-\overline{a}_{1}\zeta_{0}})\omega^{4}).
\end{array}
$$

Since $G_{1}(r,w)= G_{1}(w)$ is a polynomial of degree $3$ in the
variable $w$, it is clear that we can have at most three real
solutions for $w$ depending on the parameters of the zero-Hopf
family. Replacing these three values of $w$ in the second polynomial
$G_{2}$ we have six solutions for $r$ of the form $\pm r^{*}_{i}$
for $i=1,2,3$, because $G_{2}(r,w)$ is of the form
$P_{1}(w)r^{2}+P_{2}(w)$. However, since $r$ must be positive, we
have at most three good solutions for $G_{1}=0$ and $G_{2}=0$.
Consequently, we have at most three good solutions for
$g=(g_{1},g_{2})=0$ and then, by Theorem \ref{averagingsecond order}
and using the same arguments that in the proof of Theorem
\ref{maintheoremorigin} when we go back through the changes of
coordinates, we can have at most three limit cycles bifurcating from
the equilibrium point $p_{-}$.

\smallskip

Moreover, if we consider
$\alpha_{2}=-6\sqrt{-\overline{a}_{1}\zeta_{0}}$, then the relations
$(g_{1}(r,w),g_{2}(r,w))=(0,0)$ provide three solutions given by
$$
\begin{array}{l}
(r^{*},w^{*}_{\pm})=\vspace{0.2cm}\\
\left(\dfrac{2\omega}{\sqrt{15}}\sqrt{\dfrac{8\overline{a}_{0}
\zeta_{0}(1-\omega^{2})-\zeta_{2}}{\overline{a}_{0}^{3}\overline{a}_{1}
(\omega^2-1)}},
\pm\dfrac{\omega^{2}}{\overline{a}_{0}\sqrt{\overline{a}_{1}}}
\sqrt{-\dfrac{4\overline{a}_{0}\zeta_{0}(1-\omega^{2})+2\zeta_{2}
\omega^{2}}{5\overline{a}_{0}(\omega^2-1)}} \right)
\end{array}
$$
and
$$
(r^{0},w^{0})=\left(\dfrac{2\omega}{\sqrt{3}}\sqrt{\dfrac{4\overline{a}_{0}
\zeta_{0}(1-\omega^{2})+\zeta_{2}}{\overline{a}_{0}^{3}\overline{a}_{1}
(\omega^2-1)}} ,0\right).
$$
as long as the expressions in the square roots are positives. This
shows that three limit cycles can bifurcate simultaneous from the
equilibrium $p_-$. In a similar way we can produce examples with
one, or two limit cycles bifurcating from $p_-$. This completes the
proof of the theorem.
\end{proof}

\section*{Acknowledgments}

We thanks to Pedro T. Cardin and Tiago de Carvalho their comments
which help us to improve the presentation of this paper.

\smallskip

The first author is supported by... The second author is supported
by...  The third author is supported by the FAPESP-BRAZIL grants
2010/18015-6 and 2012/05635-1. The fourth author is partially
supported by the grants MICINN/FEDER MTM 2008--03437, AGAUR 2009SGR
410, ICREA Academia and FP7-PEOPLE-2012-IRSES-316338. All the
authors are also supported by the joint project CAPES--MECD grant
PHB-2009-0025-PC.

\end{document}